\documentclass{ws-ijnt}

\usepackage{amssymb}
\usepackage{amsmath}
\usepackage{longtable}

\newcommand{\qbin}[2]{\genfrac{[}{]}{0pt}{}{#1}{#2}}
\newcommand{\gp}[3]{\qbin{#1}{#2}_{#3}}

\newcommand{\B}{\mathcal B}

\newcommand{\C}{\mathcal C}

\newcommand{\G}{\mathcal G}
\newcommand{\GG}{\mathfrak{G}}
\newcommand{\h}{\mathcal H}
\newcommand{\hh}{\mathfrak{H}}
\newcommand{\s}{\mathcal S}
\newcommand{\sss}{\mathfrak{S}}
\newcommand{\T}{\mathcal T}
\newcommand{\TT}{\mathfrak{T}}

\allowdisplaybreaks[1]
\numberwithin{equation}{section}
\numberwithin{theorem}{section}

\begin{document}
\markboth{A. V. SILLS}{Identities of the Rogers-Ramanujan-Slater Type}

%
\catchline{}{}{}{}{}
%

\title{IDENTITIES OF THE ROGERS-RAMANUJAN-SLATER TYPE}
\author{ANDREW V. SILLS\footnote{Address as of August 1, 2007:
Department of Mathematical Sciences, Georgia Southern University,
203 Georgia Avenue, Statesboro, GA 30460-8093, USA; 
\texttt{asills@georgiasouthern.edu}.
}}

\address{Department of Mathematics, Rutgers University, 
Hill Center, Busch Campus\\
Piscataway, New Jersey 08854-8019, USA\\ 
\email{asills@math.rutgers.edu}}

\maketitle

\begin{history}
\received{(19 August 2005)}
\revised{(17 April 2006)}
\accepted{(8 August 2006)}
\comby{Bruce C. Berndt}
\end{history}

\begin{abstract}
It is shown that (two-variable generalizations of) more than half of Slater's list of 130 
Rogers-Ramanujan identities (L. J. Slater, Further identities of the Rogers-Ramanujan type, 
\emph{Proc. London Math Soc. (2)} 
\textbf{54} (1952), 147--167) can be easily derived using just three multiparameter
Bailey pairs and their associated $q$-difference equations.  As a bonus, new 
Rogers-Ramanujan type identities are found along with natural combinatorial 
interpretations for many of these identities.
\end{abstract}

\keywords{Rogers-Ramanujan identities; integer partitions; $q$-difference equations.}

\ccode{Mathematics Subject Classification 2000: 11P81, 11B65, 05A19, 39A13}

\section{Introduction}
As usual, let us define
\[ (A;q)_n :=  \prod_{j=0}^{n-1} (1-Aq^j),  \]
\[ (A;q)_\infty := \prod_{j=0}^\infty (1-Aq^j), \]
\[ (A_1, A_2, \dots, A_r; q  )_\infty := (A_{1};q)_{\infty} (A_{2};q)_{\infty} 
\cdots (A_{r};q)_{\infty},   \]
and assume throughout that $|q|<1$.

 The pair of $q$-series identities
\begin{equation}\label{euler1}
   \sum_{n=0}^\infty \frac{ q^{n^2}}{(q^2;q^2)_n} = 
   \frac{(q^2,q^2,q^4;q^4)_\infty}{(q;q)_\infty} 
 \end{equation} and
\begin{equation}\label{euler2}
 \sum_{n=0}^\infty \frac{q^{n(n+1)}}{(q^2;q^2)_n} = \frac{(q,q^3,q^4;q^4)_\infty}{(q;q)_\infty} 
\end{equation}
are special cases of one of the earliest $q$-series identities, an identity due to 
Euler~(\cite[Chapter 16]{LE},~\cite[p. 19, Eq. (2.2.6)]{GEA:top}):

\begin{equation}\label{eulerx}
  \sum_{n=0}^{\infty} \frac{ x^n q^{\binom{n}{2} }  }{(q;q)_n} = (-x;q)_\infty.
  \end{equation}
One reason that identities~\eqref{euler1} and~\eqref{euler2} are significant is that they foreshadow
a pair of deeper identities discovered first by L.~J. Rogers~\cite{LJR1894}, 
and later independently rediscovered by S.~Ramanujan~\cite[Vol. 2, p. 33 ff.]{PAM}:

\begin{theorem}[The Rogers-Ramanujan identities]
  \begin{equation}\label{RR1}
    \sum_{n=0}^\infty \frac{ q^{n^2}}{(q;q)_n} = 
       \frac{(q^2,q^3,q^5;q^5)_\infty}{(q;q)_\infty} 
   \end{equation} and
  \begin{equation}\label{RR2}
     \sum_{n=0}^\infty \frac{q^{n(n+1)}}{(q;q)_n} = \frac{(q,q^4,q^5;q^5)_\infty}
        {(q;q)_\infty}. 
   \end{equation}
\end{theorem}

Rogers~(\cite{LJR1894},\cite{LJR1917}) and Ramanujan~\cite[Chapter 11]{RLN1} 
discovered a number 
of $q$-series identities of similar type, and such identities are usually called 
``Rogers-Ramanujan type identities"
in the literature.  Another early Rogers-Ramanujan type identity is due to 
F.~H.~Jackson~\cite[p. 170, 5th eq.]{FHJ}:

\begin{theorem}[The Jackson-Slater identity]
\begin{equation}\label{JS1}
   \sum_{n=0}^\infty \frac{ q^{2n^2}}{(q;q)_{2n} } = 
   \frac{(-q^3, -q^5,q^8;q^8)_\infty}{(q^{2};q^{2})_\infty}. 
\end{equation}
\end{theorem} 
\noindent Equation~\eqref{JS1} was independently rediscovered 
by L.~J.~Slater~\cite[p. 156, Eq. (39)]{LJS1952}, who
also discovered its companion identity~\cite[p. 156, Eq. (38)]{LJS1952}:
\begin{equation}\label{JS2}
  \sum_{n=0}^\infty \frac{ q^{2n(n+1)}}{(q;q)_{2n+1} } = 
  \frac{(-q, -q^7,q^8;q^8)_\infty}{(q^{2};q^{2})_\infty}. 
\end{equation}

In a pair of papers~(\cite{WNB1},\cite{WNB2}) written during World War II, 
W.~N.~Bailey discovered the
underlying mechanism behind Rogers' identities (now known as
``Bailey's lemma"), and together with F.~J.~Dyson, 
discovered a number of new
Rogers-Ramanujan type identities.  Bailey's Ph.D. student, L.~J.~Slater, carried the 
program further and produced
a list of 130 Rogers-Ramanujan type identities.  
Although the majority of Slater's 130 identities were new, her list 
also included most of the identities 
previously discovered by Rogers, Bailey,
and Dyson, as well as a number of identities recently found to have
been recorded by Ramanujan in
his lost notebook, such as the following~\cite[p. 41]{rln}:

\begin{theorem}[The Ramanujan-Slater mod 8 identities]
\begin{equation}\label{GG1}
  \sum_{n=0}^{\infty} \frac{ q^{n^2} (-q;q^2)_n}{(q^2;q^2)_n} =
  \frac{(q^3, q^5, q^8 ; q^8)_\infty (-q;q^2)_\infty}{(q^2;q^2)_\infty}
\end{equation} and
\begin{equation}\label{GG2}
  \sum_{n=0}^{\infty} \frac{ q^{n(n+2)} (-q;q^2)_n}{(q^2;q^2)_n} =
  \frac{(q, q^7, q^8 ; q^8)_\infty (-q;q^2)_\infty}{(q^2;q^2)_\infty}.
\end{equation}
\end{theorem}
\noindent Equations~\eqref{GG1} and~\eqref{GG2} 
subsequently rose to prominence after they were interpreted combinatorially
by H. G\"ollnitz~\cite{HG} and B. Gordon~\cite{BG1965}.

 Interest in Bailey pairs, new Rogers-Ramanujan type identities, and related topics continues to the
present day; notable examples in the recent literature include papers
Andrews and Berkovich~\cite{AB}; Bressoud, Ismail, and
Stanton~\cite{BIS}; and Warnaar~\cite{SOW} to name a few.  Please see~\cite[\S 5]{AVS2005}
for a more extensive list of relevant references.

In~\cite{AVS2004} and~\cite{AVS2005}, I showed 
that all of the Bailey-Dyson results (which included many of Rogers' identities as well) 
could be recovered using a 
single device which I called the ``standard multiparameter Bailey pair" (SMBP), 
combined with a particular set of $q$-difference equations.
As a byproduct, many new double-sum Rogers-Ramanujan type identities were discovered.  
In the conclusion of~\cite{AVS2005},
I suggested that there might exist additional fundamental 
multiparameter Bailey pairs, 
which together with the SMBP, could account for not only
the Bailey-Dyson identities, but many of Slater's additional results as well.  
This in fact turns out to be the case.  Just as the two
Rogers-Ramanujan identities~(\ref{RR1}--\ref{RR2}) are the simplest results 
derivable from the SMBP,  
the Euler identities~(\ref{euler1}--\ref{euler2}), and the Jackson-Slater 
identities~(\ref{JS1}--\ref{JS2}) 
are the central results with respect to two additional multiparameter Bailey pairs which 
happen to be very closely related to the SMBP.  Accordingly, I will call these new
multiparameter Bailey pairs the ``Euler multiparameter Bailey pair (EMBP)" and the
``Jackson-Slater multiparameter Bailey pair (JSMBP)" respectively.  (See \S\ref{MBPs}
for details.)

  Once Slater's identities are viewed from the perspective of these multiparameter
Bailey pairs, it becomes clear that there are many additional (new) identities which are
``near by" Slater's results and I will present a number of
elegant examples of such identities in
\S\ref{RRS}, including 
\begin{equation}\label{ex1}
\sum_{n,r\geqq 0} \frac{q^{n^2 + 2nr +2r^2} (-q;q^2)_r }
  { (q;q)_{2r} (q;q)_{n} } 
   = \frac{(q^{10},q^{10},q^{20};q^{20})_\infty}{(q;q)_\infty},
\end{equation}
\begin{equation}\label{ex2}
\sum_{n,r\geqq 0} \frac{q^{4n^2 + 8r^2+8nr} (-q;q^2)_{2r}  }
  { (q^4;q^4)_{2r}  (q^4;q^4)_{n} } 
   = \frac{(-q^{9},-q^{11},q^{20};q^{20})_\infty}{(q^4;q^4)_\infty},
\end{equation}
and
\begin{equation}\label{ex3}
\sum_{n,r\geqq 0} \frac{q^{2n^2 + 3r^2+4nr} (-q;q^2)_{r} }
  { (-q;q)_{2n+2r} (q;q)_{2r} (q^2;q^2)_{n} } 
   = \frac{(q^{14},q^{14},q^{28};q^{28})_\infty}{(q^2;q^2)_\infty}.
\end{equation}

 But first, background information will be presented in \S\ref{Background},
followed by the definitions and examples of explicit evaluations of the multiparameter
Bailey pairs and their associated $q$-difference equations in \S\ref{MBPs}.
In \S\ref{2varRR}, the derivation of two-variable generalized Rogers-Ramanujan type identities
will be discussed.  Slater's identities and new Slater type identities
will be given in \S\ref{RRS}.  I give partition theoretic implications related
to dilated versions of infinite family partition theorems due to 
Gordon~\cite{BG}, Bressoud~\cite{DMB}, and 
Andrews-Santos~\cite{AS:attached} in \S\ref{Partitions}.  
Some observations relating to the extraction of selected summands of series
will be discussed in \S\ref{Related}.
Finally, concluding remarks and suggestions for further research appear 
in \S\ref{Conclusion}.

\section{Background}\label{Background}
\begin{definition}
A pair of sequences $\left(\alpha_n (x,q),\beta_n(x,q)\right)$ is called a  
\emph{Bailey pair} if
for $n\geqq 0$, 
   \begin{equation} \label{BPdef}
      \beta_n (x,q) = \sum_{r=0}^n \frac{\alpha_r(x,q) }{(q;q)_{n-r} (xq;q)_{n+r}}.
   \end{equation}
\end{definition}

In~\cite{WNB1} and~\cite{WNB2}, Bailey proved the fundamental
result now known as ``Bailey's lemma" (see also~\cite[Chapter 3]{GEA:qs}).

\begin{theorem}[Bailey's lemma]
If $(\alpha_n (x,q), \beta_n (x,q))$ form a Bailey pair and $N\geqq 0$, then
\begin{gather} 
  \frac{1}{ (\frac{xq}{\rho_1};q)_N ( \frac{xq}{\rho_2};q)_N} 
  \sum_{n\geqq 0} \frac{ (\rho_1;q)_n (\rho_2;q)_n 
     (\frac{xq}{\rho_1 \rho_2} ;q)_{N-n}}
     {(q;q)_{N-n}} \left( \frac{xq}{\rho_1 \rho_2} \right)^n \beta_n(x,q) 
     \nonumber\\
 = \sum_{r=0}^N \frac{ (\rho_1;q)_r (\rho_2;q)_r}
    { (\frac{xq}{\rho_1};q)_r (\frac{xq}{\rho_2};q)_r (q;q)_{N-r} (xq;q)_{N+r}}
    \left( \frac{xq}{\rho_1 \rho_2} \right)^r \alpha_r(x,q). \label{BL}
\end{gather}
\end{theorem}

It will be useful to highlight several special cases of the Bailey lemma; 
see~\cite[Cor. 1.3]{AVS2003} for more detail:

\begin{corollary}
If $(\alpha_n (x,q), \beta_n (x,q))$ form a Bailey pair, then
\begin{equation} \label{WBL}
  \sum_{n\geqq 0} x^n q^{n^2} \beta_n (x,q) 
   = \frac{1}{(xq;q)_\infty} \sum_{n=0}^\infty x^n q^{n^2} \alpha_n (x,q),
\end{equation}
\begin{equation} 
   \sum_{n\geqq 0} x^n q^{n^2} (-q;q^2)_n \beta_n (x,q^2)
  =\frac{(-xq;q^2)_\infty}{(xq^2;q^2)_\infty} 
   \sum_{n=0}^\infty \frac{ x^n q^{n^2} (-q;q^2)_n} {(-xq;q^2)_n}
    \alpha_n(x,q^2), \label{ATNSBL}
  \end{equation} and 
\begin{equation} 
   \sum_{n\geqq 0} x^n q^{n(n+1)/2} (-1;q)_n \beta_n (x,q)
  =\frac{(-xq;q)_\infty}{(xq;q)_\infty} 
   \sum_{n=0}^\infty \frac{ x^n q^{n(n+1)/2} (-1;q)_n} {(-xq;q)_{n}}
    \alpha_n (x,q). \label{SSBL}
  \end{equation} 
 \end{corollary}

\begin{remark}
The $\alpha_n (x,q)$ is normally chosen so that upon inserting it into
one of \eqref{WBL}, \eqref{ATNSBL}, \eqref{SSBL} and setting $x=1$,
the RHS sums to an infinite product via Jacobi's triple product
identity~\cite[p. 21, Thm. 2.8]{GEA:top}.  Simultaneously, the $\alpha_n (x,q)$
should be chosen so that the corresponding $\beta_n (x,q)$ is ``nice" in some
unquantifiable aesthetic sense. See also Remark~\ref{VWPrmk}.
\end{remark}

\section{Multiparameter Bailey Pairs and Associated $q$-Difference Equations}
\label{MBPs}
\subsection{Definitions of the Multiparameter Bailey Pairs}
Notice from the definition of Bailey pair~\eqref{BPdef}, once an $\alpha_n(x,q)$ 
is specified,
the corresponding $\beta_n(x,q)$ is completely determined (and vice versa).  
Thus it is sufficient to explicitly specify only one of the pair.

In~\cite{AVS2005}, I defined (with slightly different notation) the standard
multiparameter Bailey pair:
\begin{definition}[Standard multiparameter Bailey pair (SMBP)]
Let
\begin{multline} \label{SMBPdef}
\alpha_{n}^{(d,e,k)}(x,q):= 
   \frac{(-1)^{n/d} x^{(k/d-1)n/e} q^{(k/d-1+1/2d)n^2/e - n/2e} (1-x^{1/e} q^{2n/e})
   }{(1-x^{1/e}) (q^{d/e};q^{d/e})_{n/d}} \\
            \times  (x^{1/e};q^{d/e})_{n/d}\ \chi(d\mid n) ,
 \end{multline} where 
\[ \chi(P(n,d))= \left\{ 
  \begin{array}{ll}
  1 &\mbox{if $P(n,d)$ is true,}\\
  0 &\mbox{if $P(n,d)$ is false,}
  \end{array} \right.
\]
 and let $\beta_{n}^{(d,e,k)} (x,q)$ be determined by \eqref{BPdef}.
\end{definition}

Two additional multiparameter Bailey pairs will now be introduced.
\begin{definition}[Euler multiparameter Bailey pair (EMBP)]
Let
\begin{equation} \label{EMBPdef}
\widetilde{\alpha}_{n}^{(d,e,k)}(x,q):=q^{n(d-n)/2de} x^{-n/de}
\frac{(-x^{1/e};q^{d/e})_{n/d}}
{(-q^{d/e};q^{d/e})_{n/d}}  \alpha_{n}^{(d,e,k)}(x,q) 
 \end{equation} and let $\widetilde{\beta}_{n}^{(d,e,k)} (x,q)$ 
 be determined by \eqref{BPdef}.
\end{definition}

\begin{definition}[Jackson-Slater multiparameter Bailey pair (JSMBP)]
Let
\begin{equation} \label{JSMBPdef}
\bar{\alpha}_{n}^{(d,e,k)}(x,q):= (-1)^{n/d} q^{-n^2/2de}
\frac{(q^{d/2e};q^{d/e})_{n/d}}
{(x^{1/e} q^{d/2e}; q^{d/e})_{n/d}} \alpha_{n}^{(d,e,k)}(x,q) 
 \end{equation} and let $\bar{\beta}_{n}^{(d,e,k)} (x,q)$ be determined by 
\eqref{BPdef}.
\end{definition}

\begin{proposition}
\begin{multline}
\beta_n^{(d,e,k)}(x^e,q^e) = \frac{1}{(q^e;q^e)_n (x^e q^e;q^e)_n} 
   \sum_{r=0}^{\lfloor n/d \rfloor} \frac{(x;q^d)_r (1-xq^{2dr})
    (q^{-en};q^{e})_{dr}}
   {(q^d;q^d)_r (1-x) (x^e q^{e(n+1)};q^e)_{dr}}\\
   \times (-1)^{(d+1)r} x^{(k-d)r} q^{(1 +2 k - 2d - de)dr^2/2 + 
   (2en+e -1)dr/2} \label{Sbeta}
\end{multline}
\begin{multline}
\widetilde{\beta}_n^{(d,e,k)}(x^e,q^e) = \frac{1}{(q^e;q^e)_n (x^e q^e;q^e)_n} 
   \sum_{r=0}^{\lfloor n/d \rfloor} \frac{(x^2;q^{2d})_r (1-xq^{2dr}) 
   (q^{-en};q^{e})_{dr}}
   {(q^{2d};q^{2d})_r (1-x) (x^e q^{e(n+1)};q^e)_{dr}}\\
   \times (-1)^{(d+1)r} x^{(k-d-1)r} q^ {(2k - 2d - de)dr^2/2 + 
   (2en+e )dr/2} \label{Ebeta}
\end{multline}
\begin{multline}
\bar{\beta}_n^{(d,e,k)}(x^e,q^e) = \frac{1}{(q^e;q^e)_n (x^e q^e;q^e)_n} 
   \sum_{r=0}^{\lfloor n/d \rfloor} \frac{(x;q^d)_r (1-xq^{2dr}) 
   (q^{-en};q^{e})_{dr}
    (q^{d/2};q^{d})_{r}}
   {(q^d;q^d)_r (1-x) (x^e q^{e(n+1)};q^e)_{dr} (xq^{d/2};q^{d})_r}\\
   \times (-1)^{dr} x^{(k-d)r} q^ {(2k - 2d - de)dr^2/2 + 
   (2en+e-1)dr/2} \label{JSbeta}
\end{multline}
\end{proposition}
\begin{proof} Each follows directly from~\eqref{BPdef}
combined with the appropriate one of~\eqref{SMBPdef}--\eqref{JSMBPdef}.
\end{proof}

\begin{remark} \label{VWPrmk} 
The ${\alpha}_n^{(d,e,k)}$, $\widetilde{\alpha}_n^{(d,e,k)}$,
and $\bar{\alpha}_n^{(d,e,k)}$ were defined so that the series on the right hand
sides of~\eqref{Sbeta}--\eqref{JSbeta} would be \emph{very well-poised},
see, e.g., \cite{GEA:vwpt} or~\cite[Chapter 2]{GR}.
\end{remark}

\subsection{Explicit Evaluations}
\subsubsection{Selected Evaluations of the SMBP}
In Bailey~\cite[pp. 5--6]{WNB2}, several Bailey pairs are found; in our current
notation, this translates to
\begin{gather}
\beta_n^{(1,1,2)}(x,q) = \frac{1}{(q;q)_n},\\
\beta_n^{(1,2,1)}(x^2,q^2) = \frac{(-1)^n q^{n^2} }{ (q^2 ; q^2 )_n (-xq;q)_{2n}},\\
\beta_n^{(1,2,2)}(x^2,q^2) = \frac{1}{(q^2;q^2)_n (-xq;q)_{2n}},\\
\beta_n^{(1,3,2)}(x^3,q^3) = \frac{(xq;q)_{3n}}{(q^3;q^3)_n (x^3 q^3; q^3)_{2n}},\\
\beta_n^{(2,1,2)}(x,q) = \frac{q^{\binom{n}{2}}}{(q;q)_n (xq;q^2)_n},\\
\beta_n^{(2,1,3)}(x,q) = \frac{1}{(q;q)_n (xq;q^{2})_n},\\
\beta_n^{(3,1,4)}(x,q) = \frac{(x;q^3)_n}{(q;q)_n (x;q)_{2n}}.
\end{gather}
Perhaps the most important and fundamental Bailey pair of
all is the \emph{unit Bailey pair}, which appears in our current notation as
\begin{equation} \beta_n^{(1,1,1)}(x,q) = \chi({n=0}). \end{equation}
See~\cite[\S 3.5]{GEA:qs} for the standard applications of the unit Bailey pair.
 
 In~\cite{AVS2003} and \cite{AVS2005}, two dozen other explicit evaluations
are given for particular values of of $(d,e,k)$ for $\beta_n^{(d,e,k)}(x,q)$, so
we will turn our attention here to $\widetilde{\beta}_n^{(d,e,k)}(x,q)$
and $\bar{\beta}_n^{(d,e,k)}(x,q)$.  The method of evaluation employs 
standard $q$-hypergeometric summation and transformation formulas such 
as those found in~\cite{GR}; see~\cite[\S3.2]{AVS2005} for more detail.

\subsubsection{Selected Evaluations of the EMBP}
\begin{gather}
\widetilde{\beta}_n^{(1,1,1)}(x,q) = \frac{(-1)^n x^{-n}}{(q^2;q^2)_n}
   \label{E111}\\
\widetilde{\beta}_n^{(1,1,2)}(x,q) = \frac{1}{(q^2;q^2)_n}
   \label{E112}\\   
\widetilde{\beta}_n^{(1,2,2)}(x^2,q^2) = \frac{(q;q^2)_n}
   {(q^2;q^2)_n (-xq;q)_{2n}}
   \label{E122} \\
\widetilde{\beta}_n^{(1,2,3)}(x^2,q^2) = \frac{1}{(-q;q)_n^2}
    \sum_{r\geqq 0} \frac{ q^{\binom{r+1}{2}} (-x;q)_r}{(q;q)_r (-xq;q)_{n+r}(q;q)_{n-r}}
   \label{E123}  \\
\widetilde{\beta}_n^{(1,2,4)}(x^2,q^2) = 
    \sum_{r\geqq 0} \frac{ x^{2r} q^{2r^2} (q;q^2)_r}
    {(q^2;q^2)_r (-xq;q)_{2r} (q^2;q^2)_{n-r}}
   \label{E124} \\
\widetilde{\beta}_n^{(1,6,4)}(x^6,q^6) = \frac{1}{(q^6;q^6)_{2n}}
    \sum_{r\geqq 0} \frac{ x^{2r} q^{2r^2} (q;q^2)_r (x^2 q^2; q^2)_{3n-r}}
     {(q^2;q^2)_r (-xq;q)_{2r} (q^6;q^6)_{n-r}}
   \label{E164}  \\
\widetilde{\beta}_n^{(2,1,3)}(x,q) = \frac{(-q;q^2)_n }{(q^2;q^2)_n (xq;q^2)_n}
   \label{E213}   \\
   \widetilde{\beta}_n^{(2,1,4)}(x,q) = \frac{1}{(-q;q)_n}
   \sum_{r\geqq 0} \frac{ q^{\binom{r+1}{2}}(-x;q^2)_r }
   {(q;q)_r (xq;q^2)_r (q;q)_{n-r}}
   \label{E214}\\
     \widetilde{\beta}_n^{(2,1,5)}(x,q) = 
   \sum_{r\geqq 0} \frac{ q^{r^2}(-q;q^2)_r }
   {(q^2;q^2)_r (xq;q^2)_r (q;q)_{n-r}}
   \label{E215} \\
 \widetilde{\beta}_n^{(2,2,4)}(x^2,q^2) =  
 \frac{(xq^2;q^2)_n (-q^2;q^2)_n}{(x^2 q^2;q^2)_{2n}}
   \sum_{r\geqq 0} \frac{ q^{2r^2} }
   {(q^4;q^4)_r (q^4;q^4)_{n-r}}
   \label{E224} \\
   \widetilde{\beta}_n^{(2,2,5)}(x^2,q^2) =  
 \frac{1}{(-x q;q)_{2n}}
   \sum_{r\geqq 0} \frac{ x^r q^{r^2} (-q;q^2)_r  }
   {(q^2;q^2)_r (x^2 q^2;q^2)_r (q^2 ; q^2 )_{n-r}}
   \label{E225}                          
\end{gather}

\subsubsection{Selected Evaluations of the JSMBP}
\begin{gather}
\bar{\beta}_n^{(1,1,1)}(x,q) = \frac{q^{-n/2}}{(q;q)_n (x\sqrt{q};q)_n}
  \label{JS111}\\
 \bar{\beta}_n^{(1,1,2)}(x,q) = \frac{1}{(q;q)_n (x\sqrt{q};q)_n}
  \label{JS112}\\ 
\bar{\beta}_n^{(1,2,2)}(x^2,q^2) = \frac{(-x\sqrt{q};q)_{2n}}
   {(q^2;q^2)_n (-xq;q)_{2n} (x^2 q;q^2)_n} \label{JS122}\\
\bar{\beta}_n^{(1,2,3)}(x^2,q^2) = \frac{1}{(-xq;q)_{2n}}
   \sum_{r\geqq 0} \frac{x^r q^{r^2} }{(q;q)_r (x\sqrt{q};q)_r (q^2;q^2)_{n-r}}
   \label{JS123}\\
\bar{\beta}_n^{(1,2,4)}(x^2,q^2) = 
   \sum_{r\geqq 0} \frac{x^{2r} q^{2r^2} (-x\sqrt{q};q)_{2r}}
   {(-xq;q)_{2r} (x^2 q; q^2)_r (q^2;q^2)_r (q^2;q^2)_{n-r}}
   \label{JS124}\\
 \bar{\beta}_n^{(2,1,3)}(x,q) = \frac{(1-xq^n) (x;q^2)_n}{(q;q)_n (x;q)_n (xq;q^2)_n}
  \label{JS213}\\  
 \bar{\beta}_n^{(2,1,4)}(x,q) = \frac{1}{(xq;q^2)_{n}}
   \sum_{r\geqq 0} \frac{x^r q^{2r^2} }{(q;q)_{n-2r} (xq;q^2)_r (q^2;q^2)_{r}}
   \label{JS214}\\  
 \bar{\beta}_n^{(2,1,5)}(x,q) = 
   \sum_{r\geqq 0} \frac{x^r q^{r^2} (x;q^2)_r}
   {(q;q)_r (x;q)_r (xq;q^2)_r }
   \label{JS215}\\   
  \bar{\beta}_n^{(2,2,4)}(x^2,q^2) = \frac{1}{(-xq;q)_{2n}}
   \sum_{r\geqq 0} \frac{x^r q^{2r^2} }{(q^2;q^2)_{n-r}(xq;q^2)_{n-r} (xq;q^2)_r (q^2;q^2)_{r}}
   \label{JS224}\\  
  \bar{\beta}_n^{(2,2,5)}(x^2,q^2) = \frac{1}{(-xq;q)_{2n}}
   \sum_{r\geqq 0} \frac{x^r q^{r^2} (x;q^2)_r }
    {(q;q)_{r} (x;q)_{r} (xq;q^2)_r (q^2;q^2)_{n-r}}
   \label{JS225}  
\end{gather}

\subsection{$q$-Difference Equations}
In~\cite[p. 106, Eqs. (7.2.1) and (7.2.2)]{GEA:top}, Andrews defines the following functions:
\begin{definition}
For real $k$ and $i$, with $|x|<|q|^{-1}$ and $|q|<1$,
\begin{gather}
   H_{k,i}(a;x;q):= \frac{(axq;q)_\infty}{(x;q)_\infty} \sum_{n=0}^{\infty}
      \frac{x^{kn} q^{kn^2 + (1-i)n} a^n (a^{-1};q)_n  (1-x^i q^{2ni}) (x;q)_n}{ (q;q)_n (axq;q)_n }, \label{Hdef}\\
   J_{k,i}(a;x;q):= H_{k,i}(a;xq;q) - axq H_{k,i-1}(a;xq;q),\label{Jdef}
\end{gather}
\end{definition}
\noindent and proves~\cite[p. 107, Lemma 7.2]{GEA:top} that
\begin{equation}
   J_{k,i}(a;x;q)= J_{k,i-1}(a;x;q) + (xq)^{i-1} \Big( J_{k,k-i+1}(a;xq;q) - 
   a J_{k,k-i+2}(a;xq;q) \Big).
   \label{JDiffEq}
\end{equation}
Also, we will need
~\cite[p. 436, Eq. (2.8)]{GEA:qDiff}
\begin{equation} \label{HJ}
H_{k,1}(a;x;q) = J_{k,k}(a;x;q) = J_{k,k+1}(a;x;q),
\end{equation}
\cite[p. 435, Eq. (2.4)]{GEA:qDiff}
\begin{equation}\label{H0}
H_{k,0}(a;x;q) = 0,
\end{equation}
and~\cite[p. 435, Eq. (2.3)]{GEA:qDiff}
\begin{equation}\label{Hneg}
H_{k,-i}(a;x;q) = -x^{-i} H_{k,i}(a;x;q).
\end{equation}
Now we are ready to state a theorem which will yield many two-variable
Rogers-Ramanujan type identities.
\begin{theorem} For positive integers $d$, $e$, and $k$,
\begin{equation} \label{SMBPinWBL}
 \sum_{n=0}^\infty x^{en} q^{en^2} \beta^{(d,e,k)}_n (x^e,q^e) = 
    \frac{(xq^d;q^d)_\infty}{(x^e q^e; q^e)_\infty }H_{d(e-1)+k, 1}(0;x;q^d)
\end{equation}
\begin{equation} \label{EMBPinWBL}
 \sum_{n=0}^\infty x^{en} q^{en^2} \widetilde{\beta}^{(d,e,k)}_n (x^e,q^e) = 
    \frac{(xq^d;q^d)_\infty  (-x;q^d)_\infty}{(x^e q^e; q^e)_\infty }
    H_{\frac{d(e-1)+k-1}{2}, \frac 12}(0;x^2;q^{2d})
\end{equation}
\begin{equation} \label{JSMBPinWBL}
 \sum_{n=0}^\infty x^{en} q^{2en^2} \bar{\beta}^{(d,e,k)}_n (x^e,q^{2e}) = 
    \frac{(x q^{2d};q^{2d})_\infty}{(x^e q^{2e}; q^{2e})_\infty  (xq^d;q^{2d})_{\infty}}
    H_{d(e-1)+k,1}(q^{-d};x;q^{2d})
\end{equation}
\end{theorem}
\begin{proof}
Insert each of the three multiparameter Bailey pairs into~\eqref{WBL}.  The equality of each
resulting right hand side with the respective right hand sides 
of~\eqref{SMBPinWBL}--\eqref{JSMBPinWBL} is easily verified by direct calculation.
\end{proof}

  Guided by the preceding theorem, let us make the following definitions, which will
facilitate the derivation and cataloging of families of two-variable Rogers-Ramanujan
type identities.
\begin{definition}
For $d$, $e$, and $k\in \mathbb Z_+$,  
let
\begin{gather}
   Q^{(d,e,k)}_i (x):= \frac{(xq^d;q^d)_\infty }{(x^e q^e ; q^e)_\infty} J_{d(e-1)+k , i}(0;x;q^d),
     \label{QDef} \\
   \widetilde{Q}^{(d,e,k)}_i (x):= \frac{(x^2 q^{2d};q^{2d})_\infty}{(x^e q^e; q^e)_\infty }
    J_{\frac{d(e-1)+k-1}{2}, \frac i2}(0;x^2;q^{2d}) \label{QtildeDef}\\
    \bar{Q}^{(d,e,k)}_i (x):=  \frac{(x q^{2d};q^{2d})_\infty}
    {(x^e q^{2e}; q^{2e})_\infty  (xq^d;q^{2d})_{\infty}}
    J_{d(e-1)+k,i}(q^{-d};x;q^{2d}) \label{QbarDef}
\end{gather}
\end{definition}

We record the $q$-difference equations and initial conditions which establish $Q^{(d,e,k)}_i (x)$,
$\widetilde{Q}^{(d,e,k)}_i (x)$, and $\bar{Q}^{(d,e,k)}_i (x)$ uniquely as double power series in $x$ and $q$.
\begin{theorem}
For $1\leqq i \leqq k+d(e-1)$, 
  \begin{equation}\label{QDiffEq}
      Q^{(d,e,k)}_i (x) =   Q^{(d,e,k)}_{i-1} (x) +
      \frac{(1-xq^d) x^{i-1} q^{d(i-1)}}{(x^e q^e ; q^e)_{d} }Q^{(d,e,k)}_{d(e-1)+k-i+1}(xq^d)
   \end{equation}
and 
  \begin{equation}\label{QInitConds}
  Q^{(d,e,k)}_1 (0) = Q^{(d,e,k)}_2 (0) = \cdots = Q^{(d,e,k)}_{d(e-1)+k} (0)=1.
  \end{equation} 
\end{theorem}
\begin{proof}
Equation~\eqref{QDiffEq} follows from~\eqref{JDiffEq}, \eqref{QDef}, 
and \eqref{H0}.
Equation~\eqref{QInitConds} can be traced back to the fact that when $x$ is set to $0$
in~\eqref{Hdef}, the only nonzero term is the $n=0$ term, which is $1$. 
\end{proof}

The proofs of the following are analogous:
\begin{theorem}
For $1\leqq i \leqq k+d(e-1)$, 
  \begin{equation}\label{QtildeDiffEq}
      \widetilde{Q}^{(d,e,k)}_i (x) =   \widetilde{Q}^{(d,e,k)}_{i-2} (x) +
      \frac{(1-x^{2}q^{2d}) x^{i-2} q^{d(i-2)}}{(x^e q^e ; q^e)_{d} }
      \widetilde{Q}^{(d,e,k)}_{d(e-1)+k-i+1}(xq^d)
   \end{equation}
and 
  \begin{equation}\label{QtildeInitConds}
  \widetilde{Q}^{(d,e,k)}_1 (0) = \widetilde{Q}^{(d,e,k)}_2 (0) = 
    \cdots = \widetilde{Q}^{(d,e,k)}_{d(e-1)+k} (0)=1.
  \end{equation} 
\end{theorem}

\begin{theorem}
For $1\leqq i \leqq k+d(e-1)$, 
  \begin{multline}\label{QbarDiffEq}
      \bar{Q}^{(d,e,k)}_i (x) =   \bar{Q}^{(d,e,k)}_{i-1} (x) +
      \frac{(1-x q^{2d}) x^{i-1} q^{2d(i-1)}}{(x^e q^{2e} ; q^{2e})_{d} (1-xq^d)}
      \left( \bar{Q}^{(d,e,k)}_{d(e-1)+k-i+1}(xq^{2d}) \right. \\
      - \left. q^{-d}
         \bar{Q}^{(d,e,k)}_{d(e-1)+k-i+2}(xq^{2d}) \right)
   \end{multline}
and 
  \begin{equation}\label{QbarInitConds}
  \bar{Q}^{(d,e,k)}_1 (0) = \bar{Q}^{(d,e,k)}_2 (0) = 
    \cdots = \bar{Q}^{(d,e,k)}_{d(e-1)+k} (0)=1.
  \end{equation} 
\end{theorem}

\section{Two-Variable Rogers-Ramanujan Type Identities}\label{2varRR}
The basic machinery is now in place to derive many Rogers-Ramanujan 
type identities in 
the two variables $x$ and $q$.  While the $x=1$ specialization of each of the 
``$Q$" functions admits a representation
as an infinite product (see~\S \ref{RRS}),
it is advantageous to leave $x$ unspecialized for the time being.  Indeed, we shall see in
\S\ref{PtnRev} that 
many of the ``$Q$" functions enumerate certain restricted classes of integer partitions, where 
the exponent on $q$ enumerates the weight of the partition and the exponent on
$x$ enumerates the length of the partition.   
Bailey's treatment~(\cite{WNB1} and ~\cite{WNB2}), 
unlike that of Slater~(\cite{LJS1951}, \cite{LJS1952}),  
leaves this extra variable unspecialized as long as possible.  (Please note that 
in \cite{WNB1} and~\cite{WNB2}, Bailey's ``$x$" corresponds to our ``$q$" 
and Bailey's ``$a$" corresponds to our ``$x$,"
thus Bailey refers to his two variable identities as ``$a$-generalizations of 
Rogers-Ramanujan type identities."
The choice of variable name in this paper was made to simplify cross references with the 
works of Andrews~\cite{GEA:qDiff}, ~\cite{GEA:top}, ~\cite{AS:attached}.)

  We now examine how to derive many families of two-variable Rogers-Ramanujan
type identities.
\begin{definition}For positive integers $d$, $e$, and $k$, let
\begin{equation}\label{FDef}  
      {F}^{(d,e,k)}_{d(e-1)+k}(x):= 
        \sum_{n=0}^{\infty} x^{en} q^{en^2} {\beta}^{(d,e,k)}_n
         (x^e, q^e),
  \end{equation}
 \begin{equation}\label{FtildeDef}  
      \widetilde{F}^{(d,e,k)}_{d(e-1)+k}(x):= 
        \sum_{n=0}^{\infty} x^{en} q^{en^2} \widetilde{\beta}^{(d,e,k)}_n
         (x^e, q^e),
  \end{equation} and
   \begin{equation}\label{FbarDef}  
      \bar{F}^{(d,e,k)}_{d(e-1)+k}(x):= 
        \sum_{n=0}^{\infty} x^{en} q^{2en^2} \bar{\beta}^{(d,e,k)}_n
         (x^e, q^{2e}).
  \end{equation}
\end{definition}
\begin{theorem} For positive integers $d$, $e$, and $k$, 
  \begin{gather}
      F^{(d,e,k)}_{d(e-1)+k}(x) = Q^{(d,e,k)}_{d(e-1)+k}(x), \label{FQ}\\
     \widetilde{F}^{(d,e,k)}_{d(e-1)+k}(x) = \widetilde{Q}^{(d,e,k)}_{d(e-1)+k}(x),
     \label{FQtilde}\\
      \bar{F}^{(d,e,k)}_{d(e-1)+k}(x) = \bar{Q}^{(d,e,k)}_{d(e-1)+k}(x).
     \label{FQbar}
  \end{gather}
\end{theorem}
\begin{proof} First, we establish~\eqref{FQ}:
\begin{eqnarray*}
  {F}^{(d,e,k)}_{d(e-1)+k}(x)  &=& 
    \frac{(x q^{d}; q^{d})_\infty}{(x^e q^e ; q^e)_\infty} 
       H_{d(e-1)+k,1}(0;x;q^{d})
    \mbox{\ (by~\eqref{SMBPinWBL}) }\\
    &=&  \frac{(x q^{d}; q^{d})_\infty}{(x^e q^e ; q^e)_\infty} 
       J_{d(e-1)+k, d(e-1)+k}(0;x;q^{d})
    \mbox{\ (by~\eqref{HJ}) }\\
    &=& {Q}^{(d,e,k)}_{d(e-1)+k}(x) \mbox{\ (by~\eqref{QDef}).}
 \end{eqnarray*}
 Next, we establish~\eqref{FQtilde}:
 \begin{eqnarray*}
  \widetilde{F}^{(d,e,k)}_{d(e-1)+k}(x)  &=& 
    \frac{(x^2 q^{2d}; q^{2d})_\infty}{(x^e q^e ; q^e)_\infty} 
       H_{\frac{d(e-1)+k-1}{2},\frac 12}(0;x^2;q^{2d})
    \mbox{\ (by~\eqref{EMBPinWBL}) }\\
    &=&  \frac{(x^2 q^{2d}; q^{2d})_\infty}{(x^e q^e ; q^e)_\infty} 
       J_{\frac{d(e-1)+k-1}{2},\frac{d(e-1)+k-1}{2}}(0;x^2;q^{2d})
    \mbox{\ (by~\eqref{HJ}) }\\
    &=& \widetilde{Q}^{(d,e,k)}_{d(e-1)+k}(x) \mbox{\ (by~\eqref{QtildeDef}.)}
 \end{eqnarray*}
 The proof of~\eqref{FQbar} is analogous.
\end{proof}
Once the form of ${F}^{(d,e,k)}_{d(e-1)+k}(x)$,  
$\widetilde{F}^{(d,e,k)}_{d(e-1)+k}(x)$, or
$\bar{F}^{(d,e,k)}_{d(e-1)+k}(x)$ is known for a given $(d,e,k)$,  the other
expressions corresponding to $i=1,2,\dots, d(e-1)+k-1$ can be found
via the $q$-difference equations~\eqref{QtildeDiffEq}.  To keep this paper
as self contained as possible, one example will be given to illustrate this here.
Additional examples may be found in~\cite{AVS2003} 
and~\cite{AVS2004}.
\begin{example} Let us find the family of four identities associated with $(d,e,k)=(1,2,3)$
for the EMBP.
By~\eqref{FtildeDef} and~\eqref{E213} we have
\begin{eqnarray} \widetilde{F}^{(1,2,3)}_{4}(x)
&=& \sum_{n=0}^{\infty} \frac{x^{2n} q^{2n^2}}{(-q;q)_n^2}
         \sum_{r=0}^{n} \frac{q^{r(r+1)/2} (-x;q)_r}{(q;q)_r (-xq;q)_{n+r} (q;q)_{n-r}}
         \notag\\
&=& \sum_{n,r\geqq 0} \frac{x^{2n+2r} 
    q^{2n^2+4nr+ r(5r+1)/2}(-x;q)_r}
    {(-q;q)_{n+r}^2 (-xq;q)_{n+2r} (q;q)_r (q;q)_n}.   \label{Ftilde1234}      
\end{eqnarray}
Furthermore, note that we must have
\[ \widetilde{F}^{(1,2,3)}_{0}(x) = 0\] because of~\eqref{H0} and
\[ \widetilde{F}^{(1,2,3)}_{-1}(x) = -x^{-1}q^{-1} \widetilde{F}^{(1,2,3)}_1 (x)\]
because of~\eqref{Hneg}.
 
  Now we use the fact that the $\widetilde{F}^{(1,2,3)}_{i}(x)$ must satisfy 
\eqref{QtildeDiffEq}.  First, consider~\eqref{QtildeDiffEq} with $i=1$:
\[ \widetilde{F}^{(1,2,3)}_{1}(x) - \widetilde{F}^{(1,2,3)}_{-1}(x)
    = x^{-1} q^{-1} \widetilde{F}^{(1,2,3)}_{4}(xq)\] which immediately
implies
\begin{align}\label{Ftilde1231} \widetilde{F}^{(1,2,3)}_{1}(x) &= 
\frac{1}{1+xq} \widetilde{F}^{(1,2,3)}_{4}(xq)\nonumber\\
&=\sum_{n,r\geqq 0} \frac{x^{2n+2r} 
    q^{2n^2+4nr+ 2n+ 5r(r+1)/2}(-xq;q)_r}
    {(-q;q)_{n+r}^2 (-xq;q)_{n+2r+1} (q;q)_r (q;q)_n}.
 \end{align}  

 Next, consider~\eqref{QtildeDiffEq} with $i=4$, which implies
 \begin{align}
  \widetilde{F}^{(1,2,3)}_{2}(x) &=
  \widetilde{F}^{(1,2,3)}_{4}(x) - x^2 q^2 \widetilde{F}^{(1,2,3)}_{1}(xq)
  \notag\\
  &=\sum_{n,r\geqq 0} \frac{x^{2n+2r} 
    q^{2n^2+4nr+ r(5r+1)/2}(-x;q)_r}
    {(-q;q)_{n+r}^2 (-xq;q)_{n+2r} (q;q)_r (q;q)_n} \notag \\
   & \qquad -\sum_{n,r\geqq 0} \frac{x^{2n+2r+2} 
    q^{2n^2+4nr+ 2n+ r(5r+9)/2+2}(-xq;q)_{r+1}}
    {(-q;q)_{n+r}^2 (-xq;q)_{n+2r+2} (q;q)_r (q;q)_n}\notag\\
 &=\sum_{n,r\geqq 0} \frac{x^{2n+2r} 
    q^{2n^2+4nr+ r(5r+1)/2}(-x;q)_r}
    {(-q;q)_{n+r}^2 (-xq;q)_{n+2r} (q;q)_r (q;q)_n} \notag\\
   & \qquad -\sum_{n,r\geqq 0} \frac{x^{2n+2r} 
    q^{2n^2+4nr+ r(5r-1)/2+2}(-xq;q)_{r}}
    {(-q;q)_{n+r}^2 (-xq;q)_{n+2r} (q;q)_{r-1} (q;q)_n}\notag\\
  &=\sum_{n,r\geqq 0} \frac{x^{2n+2r} 
    q^{2n^2+4nr+ r(5r-1)/2}(-xq;q)_{r-1}}
   {(-q;q)_{n+r}^2 (-xq;q)_{n+2r} (q;q)_r (q;q)_n} \notag\\
   & \qquad\times\Big( q^r (1+x) - (1+xq^r)(1+q^{n+r})^2 (1-q^r) \Big). 
    \label{Ftilde1232} 
  \end{align}
Finally,~\eqref{QtildeDiffEq} with $i=2$ implies
\begin{align}
\widetilde{F}^{(1,2,3)}_{3}(x) &= \widetilde{F}^{(1,2,3)}_{2}(xq^{-1})\notag\\
&= \sum_{n,r\geqq 0} \frac{x^{2n+2r} 
    q^{2n^2+4nr-2n+ 5r(r-1)/2}(-x;q)_{r-1}}
    {(-q;q)_{n+r}^2 (-x;q)_{n+2r} (q;q)_r (q;q)_n} \notag\\
    &\qquad\quad\times\Big( q^r (1+xq^{-1}) - (1+xq^{r-1})(1+q^{n+r})^2 (1-q^r) \Big).  
    \label{Ftilde1233}
\end{align}
\end{example}    

\section{Rogers-Ramanujan-Slater Type Identities}\label{RRS}
The ``$Q$" functions simplify to infinite products when $x=1$.  This is, of course,
a crucial feature of the ``$Q$" functions since Rogers-Ramanujan type identities
establish the equivalence of $q$-series with infinite products.

\begin{lemma}
  \begin{gather}
     Q^{(d,e,k)}_i (1) = 
   \frac{(q^{di},q^{d(2ed-2d+2k+1-i)},q^{d(2ed-2d+2k+1)}; q^{d(2ed-2d+2k+1)})}
     {(q^e;q^e)_\infty}\label{QProd}\\
      \widetilde{Q}^{(d,e,k)}_i (1) = 
   \frac{(q^{di},q^{2d(de-d+k)-di},q^{2d(de-d+k)}; q^{2d(de-d+k)})}
     {(q^e;q^e)_\infty} \label{QtildeProd}\\
     \bar{Q}^{(d,e,k)}_i (1) = 
   \frac{(-q^{ d(2i-1)},-q^{4d(de-d+k)-d(2i-1)},q^{4d(de-d+k)}; q^{4d(de-d+k)})}
     {(q^{2e};q^{2e})_\infty}\label{QbarProd}
   \end{gather}
\end{lemma}
\begin{proof} Equation~\eqref{QProd} is true because
\begin{eqnarray*}
Q^{(d,e,k)}_i (1) &=& 
\frac{(q^d;q^d)_\infty }{(x^e q^e ; q^e)_\infty} J_{d(e-1)+k , i}(0;1;q^d)
\mbox{\ \ (by~\eqref{QDef})}\\
&=& \frac{(q^d;q^d)_\infty }{(x^e q^e ; q^e)_\infty}
  \frac{(q^{di}, q^{2d(de-d+k)+d-di}, q^{2d(de-d+k)+d};q^{2d(de-d+k)+d})_\infty}
  {(q^d;q^d)_\infty}\\
  &&
\mbox{\ \ (by~\cite[p. 108, Lemma 7.3]{GEA:top});}
\end{eqnarray*} the proofs equations~\eqref{QtildeProd} and~\eqref{QbarProd} 
are analogous.
\end{proof}

We are now poised to derive the majority of identities in Slater's list via the 
following procedure:
\begin{itemize}
\item Evaluate one of the three multiparameter Bailey pairs for specific values
of $d$, $e$, and $k$.  
\item Insert into one of the three limiting cases of Bailey's lemma (BL).
\item  Derive the full family of $d(e-1)+k$ associated identities via the associated
$q$-difference equations. 
\end{itemize}

Table 1 below summarizes how many of the identities in Slater's list fit into
our multiparameter Bailey pair/$q$-difference equation scheme.

The legend for Table 1 is as follows:
\begin{itemize}
  \item ``SN" stands for ``Slater number" (i.e. the equation number as
it appears in~\cite{LJS1952}).
  \item The ``MBP" column indicates which multiparameter Bailey pair (standard, Euler,
  or Jackson-Slater) is being used.
  \item ``BL" stands for ``limiting case of Bailey's lemma" with the abbreviations
``W", ``T", and ``S" representing~\eqref{WBL},~\eqref{ATNSBL}, and~\eqref{SSBL} respectively.  
   \item The ``notes" column contains references prior to
Slater, where they exist, and other explanatory comments.  
The term ``sign variant" refers to the fact that in some cases Slater used a 
Bailey pair $\alpha_{n}$ which 
amounts to the use of a particular evaluation of one of the SMBP's, but with the
extra factor of $(-1)^{n/d}$ included.  
\end{itemize}

\begin{longtable}{|r|c|c|c|c|l|}\caption{Summary of part of Slater's list}\\
\hline\hline
SN & MBP &$(d,e,k)$ & $i$& BL & Notes\\ \hline
\endfirsthead
\caption{Summary of part of Slater's list (continued)}\\
\hline\hline
SN & MBP &$(d,e,k)$ & $i$& BL & Notes\\ \hline
\endhead
\hline\multicolumn{6}{r}{\emph{Continued on next page}}
\endfoot
\hline
\endlastfoot
  1  & S & (1,1,1) & 1 & W & Euler's pentagonal number theorem\\
  2 &&&&& same as (7) with $q\to\sqrt{q}$\\
  3  & E   & (1,1,1) & 1 & W & same as Euler mod 4~\eqref{euler1} with $q\to-q$\\
  4 & E & (1,1,1) & 1 & T &\\
  6 & S & (1,1,1) & 1 & W & sign variant of (1)\\
  7 & E & (1,1,2) & 2 & W & Euler mod 4~\eqref{euler2}\\
  8 & S & (1,1,2) & 1 & S & special case of Lebesgue's identity~\cite{VAL}\\
  9 & JS & (1,1,1) & 1 & W & \\
  12 & S & (1,1,2) &   & S & special case of Lebesgue's identity~\cite{VAL}\\
 13 & & & & & linear combination of (8) and (12)\\
 14 & S  & (1,1,2) & 1 & W &Second Rogers-Ramanujan~\cite[p. 329 (2)]{LJR1894}\\
 15 & S & (1,2,1) & 1 & W & Rogers~\cite[p. 330 (5)]{LJR1917}\\
 16 & S & (1,2,2) & 1 & T & Rogers~\cite[p. 331]{LJR1894}\\
 18 & S & (1,1,2) & 2 & W &First Rogers-Ramanujan~\cite[p. 328 (2)]{LJR1894}\\
 19 & S & (1,2,1) & 2 & W & Rogers~\cite[p. 339]{LJR1894}\\
 20 & S & (1,2,2) & 2 & T & Rogers~\cite[p. 330]{LJR1894}\\
 21 &   & & & & sign variant of (20)\\
 23 & JS& (1,1,2)& 2 & T & equivalent to (3)\\
 25 & E & (1,1,2) & 2 & T &\\
 27 &  E & (1,2,2) & 1& W& Replace $q$ by $-q$.\\
 29 & JS & (1,1,2) & 1 & T &\\
 31 & S & (1,2,2) & 1 & W & Third Rogers-Selberg~\cite[p. 331 (6)]{LJR1917}\\
 32 & S & (1,2,2) & 2 & W & Second Rogers-Selberg~\cite[p. 342]{LJR1894}\\
 33 & S & (1,2,2) & 3 & W & First Rogers-Selberg~\cite[p. 339]{LJR1894}\\
 34 & S & (1,1,2) & 1 & T & Second G\"ollnitz-Gordon\\
 36 & S & (1,1,2) & 2 & T & First G\"ollnitz-Gordon\\
 38 & JS & (1,1,2) & 1 & W & Second Jackson-Slater\\
 39 & JS & (1,1,2) & 2 & W & First Jackson-Slater~\cite[p. 170, 5th Eq.]{FHJ}\\
 40 & S & (1,3,2) & 1 & W & Third Bailey mod 9~\cite[p. 422, Eq. (1.7)]{WNB1}\\
 41 & S & (1,3,2) & 3 & W & Second Bailey mod 9~\cite[p. 422, Eq. (1.8)]{WNB1}\\
 42 & S & (1,3,2) & 4 & W & First Bailey mod 9~\cite[p. 422, Eq. (1.6)]{WNB1}\\
 44 & S & (2,1,2) & 1 & W & Second Rogers mod 10~\cite[p. 330 (2)]{LJR1917}\\
 46 & S & (2,1,2) & 2 & W & First Rogers mod 10~\cite[p. 330 (2)]{LJR1917}\\
 47 & &&&& linear combination of (54) and (49)\\
 48 &&&&& linear combination of (54) and (49)\\
 49 & &&&& sign variant of (56)\\
 50 & E & (2,1,3) & 1 & W & \\
 51 & E & (2,1,3) & 2 & W & \\
 52 & S & (2,1,2) & 2 & T &\\
 53 & JS &(1,2,2) & 3 & W & \\
 54 & & & & & sign variant of (58)\\
 55 &&&&& Identity (57) with $q$ replaced by $-q$.\\
 56 & JS & (2,1,3) & 1 & W & \\
 57 & JS & (1,2,2) & 1 & W &\\
 58 & JS & (2,1,3) & 3 & W & \\
 59 & S & (2,1,3) & 1 & W & Third Rogers mod 14~\cite[p. 329 (1)]{LJR1917}\\
 60 & S & (2,1,3) & 2 & W & Second Rogers mod 14~\cite[p. 329 (1)]{LJR1917}\\
 61 & S & (2,1,3) & 3 & W & First Rogers mod 14~\cite[p. 341, Ex. 2]{LJR1894}\\
 66 &&&&& linear combination of (71) and (68)\\
 67 &&&&& linear combination of (71) and (68)\\
 68 &&&&& sign variant of (69)\\
 69 & JS & (2,1,3) & 1 & T & \\
 70 & JS & (2,1,3) & 2 & T & sign variant\\
 71 &&&&& sign variant of (72)\\
 72 & JS & (2,1,3) & 3 & T & \\
 79 & S & (2,1,3) & 3 & T & \\
 84 &&&&&identical to (9)\\
 85 & JS & (1,1,1) & 1 & W &\\
 87 &&&&& identical to (27)\\
 88 &&&&& linear combination of (90) and (91)\\
 89 &&&&& linear combination of (93) and (91)\\
 90 & S & (3,1,4) & 1 &W & Fourth Dyson mod 27\\
 91 & S & (3,1,4) & 2 & W &Third Dyson mod 27\\
 92 & S & (3,1,4) & 3 & W &Second Dyson mod 27\\
 93 & S & (3,1,4) & 4 & W & First Dyson mod 27\\
 111 &&&&& linear combination of (114) and (115)\\
 112 &&&&& linear combination of (115) and (116)\\
 113 &&&&& linear combination of (114) and (116)\\
 114 & S & (3,1,4) & 4 & T &\\
 115 & S & (3,1,4) & 3 & T &\\
 116 & S & (3,1,4) & 2 & T & \\
 \hline
\end{longtable}

The following identity follows from inserting 
$\widetilde{\beta}^{(2,1,3)}_{n}(x,q)$ into \eqref{WBL};
it is not included in Slater's list, but can be regarded as
the ``missing partner" to Slater's identities (50) and (51).
\begin{equation}\label{LJS50A}
\sum_{n=0}^\infty \frac{ q^{n^2} (-q;q^2)_n}{(q;q)_{2n}} 
  = \frac{(q^6,q^6,q^{12};q^{12})_\infty}{(q;q)_\infty}.
\end{equation}
Identity~\eqref{LJS50A} is due to Ramanujan~\cite[p. 254, Eq. (11.3.4)]{RLN1},
having been recorded by him in the Lost Notebook.

 Inserting the same Bailey pair into \eqref{ATNSBL} gives
the missing partner to Slater's (68), (70),  and (71):
\begin{equation}\label{MP70a}
\sum_{n=0}^\infty \frac{ q^{n^2} (-q^2;q^4)_n}{(q;q)_{2n} (-q^2;q^2)_{n}} 
  = \frac{(q^8,q^8,q^{16};q^{16})_\infty (-q;q^2)_\infty}{(q^2;q^2)_\infty}.
\end{equation}
To the best of my knowledge,~\eqref{MP70a} has not previously
appeared in the literature.

Several dozen double sum Rogers-Ramanujan type identities arising from the SMBP
are listed in~\cite{AVS2005}.  Below I have recorded some identities arising
from the EMBP and JSMBP.  Many more could be derived using the information
in this paper.
\begin{equation}
\sum_{n,r\geqq 0} \frac{q^{2n^2 + 4nr + r(5r+1)/2 } (-1;q)_r }
  {(-q;q)_{n+r}^2 (-q;q)_{n+2r} (q;q)_r (q;q)_n } 
   = \frac{(q^4,q^4,q^8;q^8)_\infty}{(q^2;q^2)_\infty}
   \mbox{ ($\widetilde{\beta}^{(1,2,3)}_{n}$ into \eqref{WBL})}
\end{equation}

\begin{multline}
\sum_{n,r\geqq 0} \frac{ 
    q^{2n^2+4nr+ 2n+ 5r(r+1)/2}(-q;q)_r}
    {(-q;q)_{n+r}^2 (-q;q)_{n+2r+1} (q;q)_r (q;q)_n} = 
    \frac{(q,q^7,q^8;q^8)_\infty }{(q^2;q^2)_{\infty}}\\
 \mbox{ (\eqref{Ftilde1231} with $x=1$)}
 \end{multline} 
 
 \begin{multline}
 \sum_{n,r\geqq 0} \frac{ 
    q^{2n^2+4nr+ r(5r-1)/2}(-q;q)_{r-1}}
    {(-q;q)_{n+r}^2 (-q;q)_{n+2r} (q;q)_r (q;q)_n} 
    \Big( 2q^r - (1+q^{n+r})^2 (1-q^{2r}) \Big)\\
    =\frac{(q^2,q^6,q^8;q^8)_\infty }{(q^2;q^2)_{\infty}}
    \mbox{\ \ \ (\eqref{Ftilde1232} with $x=1$)}
    \end{multline} 
    
 \begin{multline}
\sum_{n,r\geqq 0} \frac{ 
    q^{2n^2+4nr-2n+ 5r(r-1)/2}(-1;q)_{r-1}}
    {(-q;q)_{n+r}^2 (-1;q)_{n+2r} (q;q)_r (q;q)_n} 
    \Big( q^r (1+q^{-1}) - (1+q^{r-1})(1+q^{n+r})^2 (1-q^r) \Big) \\     
     =\frac{(q^3,q^5,q^8;q^8)_\infty }{(q^2;q^2)_{\infty}}
     \mbox{\ \ \ (\eqref{Ftilde1233} with $x=1$) }
    \end{multline}  
     
\begin{equation}
\sum_{n,r\geqq 0} \frac{q^{2n^2 + 4nr + 4r^2} (q;q^2)_r }
  {(q^2;q^2)_r (-q;q)_{2r} (q^2;q^2)_n } 
   = \frac{(q^5,q^5,q^{10};q^{10})_\infty}{(q^2;q^2)_\infty}
   \mbox{ ( $\widetilde{\beta}^{(1,2,4)}_{n}$ into \eqref{WBL})}
\end{equation}

\begin{multline}
\sum_{n,r\geqq 0} \frac{q^{2n^2 + 4 r^2+4nr} (-q^2;q^4)_{n+r} }
  { (-q^2;q^2)_{2n+2r} (q;q)_{2r} (q^4;q^4)_{n} } 
   = \frac{(-q^{5},-q^{7},q^{12};q^{12})_\infty (-q^2;q^4)_\infty}
   {(q^4;q^4)_\infty}\\
   \mbox{ ( $\bar{\beta}^{(1,2,3)}_{n}$ into \eqref{ATNSBL})}
\end{multline}

\begin{multline}
\sum_{n,r\geqq 0} \frac{q^{r + 2n^2 + 3r^2 + 4nr }(-q^2;q^4;n+r) (-1;q^2)_r }
  {(-q^2;q^2)_{n+r}^2 (-q^2;q^2)_{n+2r} (q^2;q^2)_r (q^2;q^2)_n } 
   = \frac{(q^6,q^6,q^{12};q^{12})_\infty (-q^2;q^4)}{(q^4;q^4)_\infty}
  \\ \mbox{\ ( $\widetilde{\beta}^{(1,2,3)}_{n}$ into \eqref{ATNSBL})}
\end{multline}

\begin{multline}
\sum_{n,r\geqq 0} \frac{q^{3n^2 + 6nr + 5r^2} (-q^3;q^6)_{n+r}
(q;q^2)_r (q^2;q^2)_{3n+2r}}
  {(q^6;q^6)_{2n+2r} (q^2;q^2)_r (-q;q)_{2r} (q^6;q^6)_n } 
   = \frac{(q^6,q^6,q^{12};q^{12})_\infty (-q^3;q^6)_\infty}{(q^6;q^6)_\infty}\\
   \mbox{ ( $\widetilde{\beta}^{(1,6,4)}_{n}$ into \eqref{ATNSBL})}
\end{multline}

\begin{equation} \label{mod16}
\sum_{n,r\geqq 0} \frac{q^{n^2 + 2r^2}   }
  { (q;q^2)_{n} (q;q)_{2r} (q;q)_{n-2r} } 
   = \frac{(-q^{7},-q^{9},q^{16};q^{16})_\infty}{(q;q)_\infty}
  \mbox{ ( $\bar{\beta}^{(2,1,4)}_{n}$ into \eqref{WBL})}
\end{equation}

\begin{multline}
\sum_{n,r\geqq 0} \frac{q^{4n^2 + 6r^2+8nr}  }
  { (-q^2;q^2)_{2n+2r} (q;q)_{2r} (q^4;q^4)_{n} } 
   = \frac{(-q^{7},-q^{9},q^{16};q^{16})_\infty}{(q^4;q^4)_\infty}\\
   \mbox{ ( $\bar{\beta}^{(1,2,3)}_{n}$ into \eqref{WBL})}
\end{multline}

\begin{multline}
\sum_{n,r\geqq 0} \frac{q^{2n^2 + 4nr + 6r^2} (-q^2;q^4)_{n+r} (-q;q^2)_{2r}  }
  { (q^4;q^4)_{2r} (q^4;q^4)_{n} } 
   = \frac{(-q^{7},-q^{9},q^{16};q^{16})_\infty (-q^2;q^4)_\infty}{(q^4;q^4)_\infty}\\
   \mbox{ ( $\bar{\beta}^{(1,2,4)}_{n}$ into \eqref{ATNSBL})}
\end{multline}

\begin{equation}
\sum_{n,r\geqq 0} \frac{q^{n^2 + 2nr + r(3r+1)/2} (-1;q^2)_r }
  {(-q;q)_{n+r} (q;q)_r (q;q)_n (q;q^2)_r } 
   = \frac{(q^8,q^8,q^{16};q^{16})_\infty}{(q;q)_\infty}
   \mbox{ ( $\widetilde{\beta}^{(2,1,4)}_{n}$ into \eqref{WBL})}
\end{equation}

\begin{multline}
\sum_{n,r\geqq 0} \frac{q^{n^2 + 2nr +3r^2} (q^2;q^2)_{n+r}  }
  { (q;q)_{2n+2r} (q^4;q^4)_{r} (q^4;q^4)_{n} } 
   = \frac{(q^{8},q^{8},q^{16};q^{16})_\infty (-q;q^2)_\infty}{(q^2;q^2)_\infty}\\
   \mbox{ ( $\widetilde{\beta}^{(2,2,4)}_{n}$ into \eqref{ATNSBL})}
\end{multline}

\begin{multline}
\sum_{n,r\geqq 0} \frac{q^{2n^2 + 4nr + 6r^2} (-q^2;q^4)_{n+r} (q^2;q^4)_{r} }
  { (-q^2;q^2)_{2r} (q^4;q^4)_{r} (q^4;q^4)_n } 
   = \frac{(q^8,q^8,q^{16};q^{16})_\infty (-q^2;q^4)_\infty}{(q^4;q^4)_\infty}\\
   \mbox{ ( $\widetilde{\beta}^{(1,2,4)}_{n}$ into \eqref{ATNSBL})}
\end{multline}

\begin{multline}
\sum_{n,r\geqq 0} \frac{q^{6n^2 + 12nr + 8r^2} (q;q^2)_r (q^2;q^2)_{3n+2r}}
  {(q^6;q^6)_{2n+2r} (q^2;q^2)_r (-q;q)_{2r} (q^6;q^6)_n } 
   = \frac{(q^9,q^9,q^{18};q^{18})_\infty}{(q^6;q^6)_\infty}\\
   \mbox{ ( $\widetilde{\beta}^{(1,6,4)}_{n}$ into \eqref{WBL})}
\end{multline}

\begin{multline}
\sum_{n\geqq 0} \frac{q^{n^2}}{(q;q)_n} \left( 
1+\sum_{r\geqq 1} \frac{q^{2r^2+2nr } (-q;q)_{r-1}}{(q;q)_r (q;q^2)_r}\right)
= \frac{(-q^9, -q^{11}, q^{20}; q^{20})_\infty}{(q;q)_\infty}\\
\mbox{ ( $\bar{\beta}^{(2,1,5)}_{n}$ into \eqref{WBL})}
\end{multline}

\begin{multline}
\sum_{n,r\geqq 0} \frac{q^{4n^2 + 8r^2+8nr} (-q;q^2)_{2r}  }
  { (q^4;q^4)_{2r}  (q^4;q^4)_{n} } 
   = \frac{(-q^{9},-q^{11},q^{20};q^{20})_\infty}{(q^4;q^4)_\infty}\\
   \mbox{ ( $\bar{\beta}^{(1,2,4)}_{n}$ into \eqref{WBL})}
\end{multline}

\begin{multline}
\sum_{n,r\geqq 0} \frac{q^{n^2 + 2nr +2r^2} (-q;q^2)_r }
  { (q;q)_{2r} (q;q)_{n} } 
   = \frac{(q^{10},q^{10},q^{20};q^{20})_\infty}{(q;q)_\infty}
   \mbox{ ( $\widetilde{\beta}^{(2,1,5)}_{n}$ into \eqref{WBL})}
\end{multline}

\begin{multline}
\sum_{n,r\geqq 0} \frac{q^{n^2 + 2r^2+2nr} (-q;q^2)_{n+r} (-q;q^2)_{r} }
  { (-q;q)_{2n+2r} (q;q)_{2r} (q^2;q^2)_{n} } 
   = \frac{(q^{10},q^{10},q^{20};q^{20})_\infty (-q;q^2)_\infty}{(q^2;q^2)_\infty}
   \\ \mbox{ ( $\widetilde{\beta}^{(2,2,5)}_{n}$ into \eqref{ATNSBL})}
\end{multline}

\begin{multline} 
\sum_{n,r\geqq 0} \frac{q^{n^2 +4r^2} (-q;q^2)_{n}   }
  { (q^2;q^4)_{n} (q^2;q^2)_{2r} (q^2;q^2)_{n-2r} } 
   = \frac{(-q^{10},-q^{14},q^{24};q^{24})_\infty (-q;q^2)_\infty}
   {(q^2;q^2)_\infty}\\
  \mbox{ ( $\bar{\beta}^{(2,1,4)}_{n}$ into \eqref{ATNSBL})}
\end{multline}

\begin{multline}\label{mod24}
\sum_{n,r\geqq 0} \frac{q^{2n^2 + 4nr +4r^2} }
  { (q;q)_{2n} (q;q)_{2r} (-q;q)_{2n+2r} } 
   = \frac{(-q^{11},-q^{13},q^{24};q^{24})_\infty}{(q^2;q^2)_\infty}\\
   \mbox{ ( $\bar{\beta}^{(2,2,4)}_{n}$ into \eqref{WBL})}
\end{multline}

\begin{multline}
\sum_{n,r\geqq 0} \frac{q^{2n^2 + 4nr +4r^2} (q^4;q^4)_{n+r} }
  { (q^2;q^2)_{2n+2r} (q^4;q^4)_{r} (q^4;q^4)_{n} } 
   = \frac{(q^{12},q^{12},q^{24};q^{24})_\infty}{(q^2;q^2)_\infty}\\
   \mbox{ ( $\widetilde{\beta}^{(2,2,4)}_{n}$ into \eqref{WBL})}
\end{multline}

\begin{multline}
\sum_{n,r\geqq 0} \frac{q^{n^2 + 2nr + 2r^2 + r} (-q;q^2)_{n+r} (-1;q^4)_r }
  {(-q^2;q^2)_{n+r} (q^2;q^2)_r (q^2;q^2)_n (q^2;q^4)_r } 
   = \frac{(q^{12},q^{12},q^{24};q^{24})_\infty (-q;q^2)_\infty}{(q^2;q^2)_\infty}
 \\  \mbox{ ( $\widetilde{\beta}^{(2,1,4)}_{n}$ into \eqref{ATNSBL})}
\end{multline}

\begin{multline}
\sum_{n\geqq 0} \frac{q^{2n^2}}{(q^2;q^2)_n} \left( 
\frac{1}{(-q;q)_{2n}}
+\sum_{r\geqq 1} \frac{q^{3r^2+4nr } (-q;q)_{r-1}}{(q;q)_r (q;q^2)_r (-q;q)_{2n+2r}}
\right)\\
= \frac{(-q^{13}, -q^{15}, q^{28}; q^{28})_\infty}{(q^2;q^2)_\infty}
\mbox{\ \ \ ( $\bar{\beta}^{(2,2,5)}_{n}$ into \eqref{WBL})}
\end{multline}

\begin{equation}
\sum_{n,r\geqq 0} \frac{q^{2n^2 + 3r^2+4nr} (-q;q^2)_{r} }
  { (-q;q)_{2n+2r} (q;q)_{2r} (q^2;q^2)_{n} } 
   = \frac{(q^{14},q^{14},q^{28};q^{28})_\infty}{(q^2;q^2)_\infty}
   \mbox{ ( $\widetilde{\beta}^{(2,2,5)}_{n}$ into \eqref{WBL})}
\end{equation}

\begin{multline}
\sum_{n,r\geqq 0} \frac{q^{n^2 + 2nr +3r^2} (-q;q^2)_{n+r} (-q^2;q^4)_r }
  { (q^2;q^2)_{2r} (q^2;q^2)_{n} } 
   = \frac{(q^{16},q^{16},q^{32};q^{32})_\infty (-q;q^2)_\infty}
   {(q^2;q^2)_\infty}\\
   \mbox{ ( $\widetilde{\beta}^{(2,1,5)}_{n}$ into \eqref{ATNSBL})}
\end{multline}

\section{Partition Identities}\label{Partitions}

\subsection{Review of Definitions, Notation, and Classical Infinite Family Partition Identities}
\label{PtnRev}
\begin{definition} A \emph{partition} $\pi$ of an integer $n$ is a nonincreasing 
finite sequence 
$\pi$ of positive integers 
$ ( \pi_1, \pi_2, \dots, \pi_s )$ such that 
$\sum_{h=1}^s \pi_h=n.$  Each $\pi_h$ is called a \emph{part} of the partition
$\pi$.  If $\pi=( \pi_1, \dots, \pi_s )$ is a partition of $n$, then $s$ is called
the \emph{length} of $\pi$ and $n$ is called the \emph{weight} of $\pi$.  The
\emph{multiplicity} of the positive integer $r$ in the partition $\pi$, $m_r(\pi)$, 
is the number
of times $r$ appears as a part of $\pi$, i.e.
$m_r (\pi) = \#\{ h : \pi_h = r \}.$
\end{definition}

\begin{definition} An \emph{overpartition} of an integer $n$ is a 
nonincreasing finite sequence of positive integers whose sum is $n$ in which
the first occurrence of a given number may be overlined.
\end{definition}
\noindent Thus, while there are only three partitions of $3$, namely $(3)$, $(2,1)$, and
$(1,1,1)$; there are eight overpartitions of $3$, namely $(3)$, $(\bar{3})$,
$(2,1)$, $(\bar{2},1)$, $(2,\bar{1})$, $(\bar{2}, \bar{1})$, $(1,1,1)$
$(\bar{1},1,1)$.

The Rogers-Ramanujan identities were first interpreted combinatorially by
MacMahon~\cite{PAM} and Schur~\cite{IS}.  
In 1961, Basil Gordon~\cite{BG} embedded the
combinatorial Rogers-Ramanujan identities as the $k=2$ case of the
following infinite family of partition identities:
\begin{theorem}[Gordon's combinatorial generalization of Rogers-Ramanujan]
Let $1\leqq i\leqq k$.
Let $\G_{k,i}(n)$ denote the number of partitions $\pi$ of $n$ such that 
$m_1(\pi)\leqq i-1$ and $\pi_{j}-\pi_{j+k-1}\geqq 2$.  
Let $\h_{k,i}(n)$ denote the number of partitions of $n$ into parts 
$\not\equiv 0,\pm i\pmod{2k+1}$.  Then $\G_{k,i}(n) = \h_{k,i}(n)$ for all $n$.
\end{theorem}

\begin{remark}
When comparing to the Andrews-Santos theorem below, it will be useful to
note that the second condition of $\G_{k,i}(n)$ is equivalent to
``$m_{j}(\pi) + m_{j+1}(\pi)\leqq k-1$ for any positive integer $j$."
\end{remark}

Just as the combinatorial Rogers-Ramanujan identities are the $k=2$ cae of
Gordon's theorem, combinatorial versions of the Euler identities~\eqref{euler1} and
~\eqref{euler2} form the $k=2$ case of the theorem of
D.~Bressoud~\cite{DMB}.

\begin{theorem}[Bressoud's even modulus partition identity]
\label{BressThm}
 Let $1\leqq i <  k$.
Let $\B_{k,i}(n)$ denote the number of partitions $\pi$ of $n$ such that 
\begin{itemize}
    \item $m_1(\pi) \leqq i-1$,
    \item $\pi_j - \pi_{j+k-1} \geqq 2,$ and
    \item if $\pi_j - \pi_{j+k-2} \leqq 1$, 
          then $\sum_{h=0}^{k-2} \pi_{j+h} \equiv (i-1)\pmod{2}.$ 
 \end{itemize}
Let $\C_{k,i}(n)$ denote the number of partitions of $n$ into parts 
$\not\equiv 0, \pm  i\pmod{2k}$.
Then $\B_{k,i}(n) = \C_{k,i}(n)$ for all $n$.
\end{theorem}

\begin{remark} Theorem~\ref{BressThm} in the case where $k$ is odd
is actually due to Andrews~(\cite[p. 432, Thm. 2]{GEA:jct} and 
\cite[p. 117, Ex. 8]{GEA:top}).
\end{remark}

\begin{remark} Notice that Bressoud's identity, unlike the other identities
listed here, excludes the case $i=k$.  This is presumably because
  \begin{equation}\label{BressoudMiddleProduct}
    \sum_{n=0}^{\infty} \C_{k,k}(n) q^n = 
    \frac{(q^k, q^{k}, q^{2k}; q^{2k})}{(q;q)_\infty},
  \end{equation}
the right hand side of which had no nice combinatorial interpretation 
known at the time of Bressoud's work.  (The difference condition arguments work
fine in the $i=k$ case.)  However, in light of some recent work by 
Andrews and Lewis~\cite{AL}, we may now given an elegant combinatorial
interpretation of the infinite product in~\eqref{BressoudMiddleProduct}.

  Andrews and Lewis~\cite[p. 79, Eq. (2.2)]{AL} showed that when
$0<a<b<k$, the infinite product
 \[  \frac{(q^{a+b};q^{2k})_{\infty}}{(q^a, q^b; q^k)_{\infty}}\]
generates partitions in which all parts are congruent to $a$ or $b$ modulo $k$, but
that for any $j$, $kj+a$ and $kj+b$ are not both parts.

Consequently, by observing that for $k\geqq 3$
\begin{equation}
\sum_{n=0}^{\infty} \C_{k,k}(n) q^n = 
\frac{(q^k, q^k, q^{2k}; q^{2k})}{(q;q)_\infty}
=\frac{(q^k;q^{2k})_\infty}{(q,q^{k-1};q^k)_\infty}\cdot
\frac{1}{(q^2, q^3, \cdots, q^{k-2};q^k)_\infty},
\end{equation}
(where the second factor on the RHS represents the empty product in the case $k=3$),
we may make the following definition:
\begin{definition}\label{BressMidDef}
Let $\C_{k,k} (n)$ denote the number of partitions of $n$ such that for any nonnegative
integer $j$, the integers $kj+2$, $kj+3$, \dots, $kj+(k-2)$ may appear as parts.  Additionally,
for any nonnegative integer $j$, either $kj+1$ or $ kj+(k-1)$, but not both,
may appear as parts.  Furthermore, by standard elementary
techniques, it is natural to define $\C_{2,2} (n)$ to be the number of partitions of $n$ into distinct odd parts.
\end{definition}
Now we may naturally extend Bressoud's identity to the case $i=k$:
\begin{theorem}\label{BressExt} 
For $k\geqq 2$, let $\B_{k,k}(n)$ denote the number of partitions
 $\pi$ of $n$ such that 
\begin{itemize}
    \item $m_1(\pi) \leqq k-1$,
    \item $\pi_j - \pi_{j+k-1} \geqq 2,$ and
    \item if $\pi_j - \pi_{j+k-2} \leqq 1$, 
          then $\sum_{h=0}^{k-2} \pi_{j+h} \equiv (k-1)\pmod{2}.$ 
 \end{itemize}
Let $\C_{k,k}(n)$ be as in Definition~\ref{BressMidDef}.  
Then $\B_{k,k}(n) = \C_{k,k}(n).$
\end{theorem}
\end{remark}

In order to facilitate the statement of the next result, we make the following 
definition.
\begin{definition}  Let $n$, $k$ and $i$ be specified positive integers with 
$1\leqq i \leqq k$, 
and let $\pi$ be a partition of $n$.
The \emph{modified presentation} $\hat{\pi}$ of the partition $\pi$
is the finite sequence obtained by appending $k-i$ zeros to $\pi$.
\end{definition}
\begin{theorem}[Andrews-Santos combinatorial generalization of Jackson-Slater]
 Let $1\leqq i \leqq  k$.
Let $\s_{k,i}(n)$ denote the number of partitions $\pi$ of $n$ such that 
\begin{itemize}
    \item $m_2(\pi) \leqq i-1$,
    \item For any positive integer $j$, $m_{2j}(\pi) + m_{2j+2}\leqq k-1$,
    \item For any positive integer $j$, $m_{2j-1}(\pi)>0$ only if
      $m_{2j-2}(\hat{\pi}) + m_{2j}(\pi) = k-1$. 
 \end{itemize}
Let $\T_{k,i}(n)$ denote the number of partitions of $n$ into parts 
which are either even but not multiples of $4k$, or 
distinct, odd, and congruent to $\pm(2i-1)\pmod{4k}$.
Then $\s_{k,i}(n) = \T_{k,i}(n)$ for all $n$.
\end{theorem}

\subsection{Dilations}
\subsubsection{The Dilated Gordon Theorem}
In~\cite{AVS2003}, I showed that the 
Rogers-Ramanujan type identities which arose by inserting 
$\Big(  \alpha^{(d,1,k)}_n(x,q) , \beta^{(d,1,k)}_n(x,q) \Big)$
into~\eqref{WBL} and considering the associated $q$-difference equations, could
be interpreted combinatorially as the generating functions for partitions counted
in Gordon's theorem, dilated by a factor of $d$, with nonmultiples of $d$ allowed
to appear without restriction.  
  
  In the present notation, this amounts to the observation that
  \begin{eqnarray} 
       Q^{(d,1,k)}_i(x) &=& 
\frac{(xq^d;q^d)_\infty}{(xq;q)_\infty} J_{k,i}(0;x;q^d)
       \mbox{\ \ (by~\eqref{QDef})}\notag \\
      &=&   \frac{(xq^d; q^d)_\infty}{(xq;q)_\infty} 
         \sum_{n=0}^\infty \sum_{m=0}^n 
         \G'_{k,i}(m,n) x^m q^{dn}
\mbox{\ \ (by~\cite[p. 109, Thm. 7.5]{GEA:top})},
      \notag
  \end{eqnarray}
where $\G'_{k,i} (m,n)$ denotes the number of partitions of $n$ enumerated by
$\G_{k,i} (n)$ which have exactly $m$ parts.

More formally, we have the following theorem:
\begin{theorem} Let $d$ be a positive integer and
let $1\leqq i\leqq k$.
Let $\GG_{d,k,i}(n)$ denote the number of partitions $\pi$ of $n$ such that 
$m_d (\pi)\leqq i-1$ and 
$m_{dj}(\pi) + m_{d(j+1)}(\pi)\leqq k-1$ for any positive integer $j$. 
Let $\hh_{d,k,i}(n)$ denote the number of partition of $n$ into parts 
$\not\equiv 0,\pm di\pmod{2d(k+1)}$.  
Then $\GG_{d,k,i}(n) = \hh_{d,k,i}(n)$ for all $n$.  Furthermore,
this identity is a combinatorial interpretation of the analytic
identity $F^{(d,1,k)}_{i}(1) = Q^{(d,1,k)}_{i}(1)$.
\end{theorem}

For example, insert the $(d,e,k)=(3,1,4)$ case of the standard multiparameter
Bailey pair into~\eqref{WBL}.  Then
use the system of $q$-difference equations to find $F^{(3,1,4)}_3 (x)$.  The
resulting identity $F^{(3,1,4)}_3 (1)  = Q^{(3,1,4)}_3 (1)$
 turns out to be one of Dyson's identities~\cite[p.  433, Eq. (B3)]{WNB1},
 \begin{equation}
  \sum_{n=0}^{\infty} \frac{q^{n(n+1)} (q^3;q^3)_n}{(q;q)_n (q;q)_{2n+1}}
  = \frac{(q^9;q^9)_\infty}{(q;q)_\infty},
\end{equation}
which in turn can be interpreted combinatorially in terms of the $k=4$, $i=3$ case of
Gordon's theorem dilated by a factor of 3, i.e.
   ``the number of partitions of $n$ into nonmultiples of 9 equals the number of 
 partitions of $n$ wherein 3 can appear as a part at most twice, the total
 number of appearances of any two consecutive multiples of 3 is at most three,
 and 1's and 2's can appear without restriction."
 
 \subsubsection{The Dilated Bressoud Theorem}
     From the preceding sections, we see that
 \begin{eqnarray} 
       \widetilde{Q}^{(d,1,k)}(x) &=& 
       \frac{(xq^{d};q^{d})_\infty}{(xq;q)_\infty} 
       (-xq^d;q^d)_{\infty} J_{(k-1)/2,i/2}(0;x^2;q^{2d})
       \mbox{\ \ (by~\eqref{QtildeDef})}\notag \\
      &=&   \frac{(xq^d; q^d)_\infty}{(xq;q)_\infty} 
         \sum_{n=0}^\infty \sum_{m=0}^n 
         \B'_{k,i}(m,n) x^m q^{dn}\mbox{\ \ (by~\cite[p. 67, 1st Eq.]{DMB})},
      \notag
  \end{eqnarray}
where $\B'_{k,i} (m,n)$ 
denotes the number of partitions of $n$ enumerated by
$\B_{k,i} (n)$ which have exactly $m$ parts.  Accordingly the
identities $\widetilde{F}^{(d,1,k)}_{i}(1) = \widetilde{Q}^{(d,1,k)}_{i}(1)$
can be interpreted combinatorially as partitions wherein nonmultiples of $d$ allowed to appear
as parts without restriction, and multiples of $d$ may appear subjected to
the restrictions imposed by Bressoud's theorem (dilated by a factor of $d$).

\subsubsection{The Dilated Andrews-Santos Theorem}
 The situation with the dilated version of the Andrews-Santos theorem
is somewhat less straight forward, and will have to be broken down to two cases.
\begin{eqnarray} 
       \bar{Q}^{(d,1,k)}(x) &=& 
       \frac{(xq^{2d};q^{2d})_\infty}{(xq^{2};q^{2})_\infty} 
        \frac{1}{(xq^d;q^{2d})_{\infty}} J_{k,i}(q^{-d};x;q^{2d})
       \mbox{\ \ (by~\eqref{QbarDef})}\notag \\
      &=&   \frac{(xq^{2d}; q^{2d})_\infty}{(xq^2;q^2)_\infty} 
         \sum_{n=0}^\infty \sum_{m=0}^n 
         \s'_{k,i}(m,n) x^m q^{dn}\\
         & & \qquad\quad \mbox{\ \ (by~\cite[p. 96, (5.4)]{AS:attached})},
         \label{DilAS}
  \end{eqnarray}
where $\s'_{k,i} (m,n)$ 
denotes the number of partitions of $n$ enumerated by
$\s_{k,i} (n)$ which have exactly $m$ parts.

  Since ${(xq^{2d}; q^{2d})_\infty}/{(xq^2;q^2)_\infty}$ enumerates partitions
into parts which are even but not multiples of $2d$, and 
$ \sum_{n=0}^\infty \sum_{m=0}^n \s'_{k,i}(m,n) x^m q^{dn}$ enumerates
``Andrews-Santos" partitions dilated by a factor of $d$, when $d$ is odd these two
sets of partitions involve parts of independent
magnitudes, but when $d$ is even 
there is a conflict as both the
infinite product and the double sum generate parts congruent to 
$d\pmod {2d}$.  We will consider the easy case, where $d$ is odd, first.  By
\eqref{DilAS}, we have already proved the following theorem:
\begin{theorem} Let $d$ be an odd positive integer and
 let $1\leqq i \leqq  k$.
Let $\sss_{d,k,i}(n)$ denote the number of partitions $\pi$ of $n$ such that 
\begin{itemize}
    \item $m_{2d}(\pi) \leqq i-1$,
    \item for any positive integer $j$, $m_{2dj}(\pi) + m_{2dj+2d}\leqq k-1$,
    \item for any positive integer $j$, $m_{2dj-d}(\pi)>0$ only if
      $m_{2dj-2d}(\hat{\pi}) + m_{2dj}(\pi) = k-1$. 
 \end{itemize}
Let $\TT_{d,k,i}(n)$ denote the number of partitions of $n$ into parts 
which are either even but not multiples of $4dk$, or 
distinct, odd, and congruent to $\pm d(2i-1)\pmod{4dk}$.
Then $\sss_{d,k,i}(n) = \TT_{d,k,i}(n)$ for all $n$.
Furthermore,
this identity is a combinatorial interpretation of the analytic
identity $\bar{F}^{(d,1,k)}_{i}(1) = \bar{Q}^{(d,1,k)}_{i}(1)$.
\end{theorem}

 Next, we turn to the case where $d$ is even.  Notice that by standard elementary 
reasoning,
\[ \bar{Q}^{(d,1,k)}_i (1) = 
\frac{(-q^{d(2i-1)},-q^{4dk-d(2i-1)},q^{4dk};q^{4dk})_\infty}
{(q^2;q^2)_\infty}\]
can be interpreted combinatorially as the generating 
function for overpartitions of $n$ where the
overlined parts are congruent to $\pm d(2i-1)\pmod{4dk}$ and the nonoverlined parts
are even but not multiples of $4dk$.  Further, note that
\[  \frac{(q^{2d}; q^{2d})_\infty}{(q^2;q^2)_\infty}
   = \frac{1}{(q^2;q^{2d})_\infty (q^4;q^{2d})_\infty
   \cdots(q^{2d-2}; q^{2d})_\infty} \]
must contain the factor $(q^d;q^{2d})_\infty^{-1}$ when $d$ is even.  
This factor
can be interpreted as generating partitions into parts which are odd 
and such that
if a given odd number appears, the number of appearances must 
be a multiple of $d$.
The other factors 
\[ \frac{1}{(q^2;q^{2d})_\infty (q^4;q^{2d})_\infty
\cdots (q^{d-2};q^{2d})_\infty
(q^{d+2};q^{2d})_\infty (q^{d+4};q^{2d})_\infty
\cdots (q^{2d-2}; q^{2d})_\infty}, \]
generate partitions into parts which are even but $\not\equiv 0,d\pmod{2d}$.
Then the $x=1$ case of~\eqref{DilAS} allows us to state the following
theorem:
\begin{theorem} Let $d$ be an even positive integer and
 let $1\leqq i \leqq  k$.
Let $\sss_{d,k,i}(n)$ denote the number of partitions $\pi$ of $n$ such that 
\begin{itemize}
    \item for any positive integer $j$, $d \mid m_{2j-1}(\pi)$,
    \item $m_{2d}(\pi) \leqq i-1$,
    \item for any positive integer $j$, $m_{2dj}(\pi) + m_{2dj+2d}\leqq k-1$,
    \item for any positive integer $j$, $m_{2dj-d}(\pi)>0$ only if
      $m_{2dj-2d}(\hat{\pi}) + m_{2dj}(\pi) = k-1$. 
 \end{itemize}
Let $\TT_{d,k,i}(n)$ denote the number of overpartitions of $n$ into parts where
which are either even but not multiples of $4dk$, and 
the overlined parts are congruent to $\pm d(2i-1)\pmod{4dk}$.
Then $\sss_{d,k,i}(n) = \TT_{d,k,i}(n)$ for all $n$.
Furthermore,
this identity is a combinatorial interpretation of the analytic
identity $\bar{F}^{(d,1,k)}_{i}(1) = \bar{Q}^{(d,1,k)}_{i}(1)$.
\end{theorem}

  Thus, for example, since~\eqref{mod16} with $q\to q^2$ is
$\bar{Q}^{(2,1,4)}_{4}(1)$, its combinatorial interpretation is
``the number of partitions $\pi$ of $n$ into parts such that
\begin{itemize}
  \item any odd number appears as a part an even number of times,
  \item 4 appears as a part at most three times,
  \item for any positive integer $j$, $m_{4j}(\pi) + m_{4j+4} \leqq 3$, and
  \item for any positive integer $j$, $m_{4j-2}>0$ only if
    $m_{4j-4}(\hat{\pi})+m_{4j}(\pi) = 3$,
\end{itemize}
equals the number of overpartitions of $n$ such that
  \begin{itemize}
    \item overlined parts are congruent to $\pm (4i-2) \pmod{32}$, and
    \item nonoverlined parts are even but not multiples of $32$."
  \end{itemize}
  
\begin{remark}
When $e>1$, the series $F^{(d,e,k)}_{i}(x)$,  $\widetilde{F}^{(d,e,k)}_{i}(x)$,
and $\bar{F}^{(d,e,k)}_{i}(x)$ do not, in general, have nonnegative coefficients,
so there is no hope of interpreting them directly in terms of ordinary partitions or
overpartitions.
It would be interesting, however, to find combinatorial interpretations in terms of objects
more general than ordinary partitions or overpartitions.
\end{remark}

\section{Extraction of Selected Indices}\label{Related}
Some identities can be thought of as the result of extracting half of the summands
of another more primitive identity.

For example consider~\eqref{mod24},
\begin{eqnarray}
\frac{(-q^{11},-q^{13},q^{24};q^{24})_\infty}{(q^2;q^2)_\infty} &=&
\sum_{n,r\geqq 0} \frac{q^{2n^2 + 4nr +4r^2} }
  { (q;q)_{2n+2r} (q;q)_{2r} (-q;q)_{2n+2r} } \notag\\
&=& \sum_{N=0}^\infty \sum_{r=0}^N \frac{q^{2N^2+2r^2}}{(q^2;q^2)_{2N}}
\gp{2N}{2r}{q} \mbox{\ (by taking $N=n+r$)} \notag\\
&=& \sum_{N=0}^\infty \frac{q^{2N^2}}{(q^2;q^2)_{2N}}
\sum_{r=0}^{2N} q^{r^2/2}\gp{2N}{2r}{q} \chi(2\mid r) \notag\\
&=& \sum_{N=0}^\infty \frac{q^{2N^2}}{2(q^2;q^2)_{2N}}
\sum_{r=0}^{2N} q^{r^2/2}\gp{2N}{2r}{q} \left( 1 + (-1)^r \right) \notag\\
&=& \frac{1}{2} \sum_{n=0}^\infty \frac{q^{2n^2}}{(q^2;q^2)_{2n}}
 (\sqrt{q};q)_{2n} + (-\sqrt{q};q)_{2n}) 
\mbox{ \ (by~\cite[p. 490, Cor. 10.2.2(c)]{AAR}) }\notag \\
&=& \frac{1}{2} \left(
  \sum_{n=0}^\infty \frac{q^{2n^2} (\sqrt{q};q)_{2n}}{(q^2;q^2)_{2n}}
 +\sum_{n=0}^\infty \frac{q^{2n^2} (-\sqrt{q};q)_{2n}}{(q^2;q^2)_{2n}}
 \right),
 \label{splitSL53half}
\end{eqnarray}
where 
\[  \gp{A}{B}{q}:=  \left\{
  \begin{array}{ll} 
    \frac{(q;q)_A}{(q;q)_B (q;q)_{A-B}} &\mbox{if $0\leqq A\leqq B$},\\
    0 & \mbox{otherwise.} 
  \end{array}
 \right. \]

Consider the expression
\begin{equation}\label{splitSL53}
\frac{1}{2} \left(
  \sum_{n=0}^\infty \frac{q^{4n^2} (q;q^2)_{2n}}{(q^4;q^4)_{2n}}
 +\sum_{n=0}^\infty \frac{q^{4n^2} (-q;q^2)_{2n}}{(q^4;q^4)_{2n}}
 \right),
 \end{equation}
which is merely~\eqref{splitSL53half} with $q$ replaced by $q^2$.  
The first summand of~\eqref{splitSL53} 
is the series expression from one of Slater's identities~\cite[p. 157, Eq. (53)]{LJS1952},
while the second summand is the same as the first with $q$ replaced by $-q$.
Thus~\eqref{mod24} may be thought of as an ``average" of two 
complementary presentations of Slater's identity number 53 (which we obtain here
by inserting $\bar{\beta}^{(1,2,2)}_ n(x,q)$ into~\eqref{WBL}).

 This line of inquiry can be pushed further.  Let us examine Slater's identity 
number 53 similarly.
\begin{eqnarray}
\sum_{n=0}^\infty \frac{q^{4n^2} (-q;q^2)_{2n}}{(q^4;q^4)_{2n}}
&=&\sum_{n=0}^\infty \frac{q^{n^2} (-q;q^2)_{n}}{(q^4;q^4)_{n}} \chi(2\mid n)
\notag\\
&=&\sum_{n=0}^\infty \frac{q^{n^2} (-q;q^2)_{n}}{(q^4;q^4)_{n}} \left(
 \frac{(-1)^n+1}{2} \right)\notag\\
&=& \frac{1}{2}\left(
\sum_{n=0}^\infty \frac{(-1)^n q^{n^2} (-q;q^2)_{n}}{(q^4;q^4)_{n}}+
\sum_{n=0}^\infty \frac{q^{n^2} (-q;q^2)_{n}}{(q^4;q^4)_{n}}\right).
\label{splitSL4}
\end{eqnarray}
The left summand of~\eqref{splitSL4} is Slater's identity number 4 (which is
obtained here by inserting $\widetilde{\beta}^{(1,1,1)}_n (x,q)$
into~\eqref{ATNSBL}), while the right summand is a sign variant of same.

  This phenomenon is by no means isolated.  In fact, in a different context, 
Andrews observed~\cite[p. 222]{KA} that Slater's identities 98 and 94, 
dilated by $q\to q^4$, are the even and odd summands respectively of Slater's
identity 20, i.e.,
\begin{eqnarray*}
\sum_{n=0}^\infty \frac{q^{n^2}}{(q^4;q^4)_n}
&=& \sum_{n=0}^\infty \frac{q^{n^2}}{(q^4;q^4)_n}\chi(2\mid n)
       +\sum_{n=0}^\infty \frac{q^{n^2}}{(q^4;q^4)_n}\chi(2\nmid n)\\
 &=& \sum_{n=0}^\infty \frac{q^{(2n)^2}}{(q^4;q^4)_{2n} }
    +  \sum_{n=0}^\infty \frac{q^{(2n+1)^2}}{(q^4;q^4)_{2n+1} }\\
 &=& \sum_{n=0}^\infty \frac{q^{4n^2}}{(q^4;q^4)_{2n} }
    +  q\sum_{n=0}^\infty \frac{q^{4n(n+1)}}{(q^4;q^4)_{2n+1} }.
\end{eqnarray*}
Analogously, Slater's identities 99 and 96 with $q\to q^4$ 
are the even and odd summands, respectively, of identity 16.  Thus in this instance, the
four identities (Slater's 94, 96, 98, 99) are \emph{not} derivable directly from any of
the three multiparameter Bailey pairs of this paper, but are nonetheless a simple 
consequence of two of the identities (Slater's 16 and 20) which \emph{are} derivable
from the SMBP.

\begin{remark}
All six of these identities (Slater's 16, 20, 94, 96, 98, 99) are actually due to
L. J. Rogers~\cite{LJR1894}.
\end{remark}

\section{Conclusion}\label{Conclusion}
We have seen (Table 1) that 71 of Slater's 130 identities can be explained by the
three multiparameter Bailey pairs combined with limiting cases of
Bailey's lemma and their associated $q$-difference equations.
It is reasonable to ask why the other 59 identities do not fit into this scheme.
There seem to be two fundamental characteristics that distinguish the 71 
from the remaining 59. 

  Firstly, the variation of~\eqref{SSBL} actually used by 
Slater~(\cite{LJS1951},
\cite{LJS1952}) is obtained by setting $x=q$, $\rho_1=-q$ and letting 
$N, \rho_2\to\infty$
in~\eqref{BL} to obtain
\begin{equation} \label{SSBL2}
\frac{1}{1-q}\sum_{n\geqq 0} q^{n(n+1)/2} (-q;q)_n \beta_n (q,q)
  =\frac{(-q;q)_\infty}{(q;q)_\infty} 
   \sum_{n=0}^\infty  q^{n(n+1)/2} 
    \alpha_n (q,q).
\end{equation}
With few exceptions, a
Bailey pair $\alpha_{n}$ which fits nicely into~\eqref{WBL} and~\eqref{ATNSBL}
to give an instance of the triple product identity when $x=1$ or $x=q$ is likely \emph{not}
to do so when inserted into~\eqref{SSBL2}.  Thus there may be one or more
multiparameter Bailey pairs which work well with~\eqref{SSBL2}.
Combinatorially, identities associated with~\eqref{SSBL2} fit naturally into
the realm of overpartitions, as studied recently by 
Jeremy Lovejoy~(\cite{JL2003},~\cite{JL2004}), sometimes jointly with
Sylvie Corteel~\cite{CL}.
 
  The other distinguishing feature in some of the 59 identities not accounted
for in this paper is the appearance of
instances of the quintuple product identity~\cite{GNW} in the product side.
With further study, it is hoped that these identities can be understood 
in a manner similar
to those already explained here.

\section*{Acknowledgments}
The author gratefully acknowledges the referee for catching a number of
subtle misprints, and for assistance with the exposition.

\end{document}